\documentclass[a4paper,twoside]{amsart}
\usepackage[hmargin={25mm,25mm},vmargin={25mm,25mm}]{geometry}
\bibstyle{plain}

\newtheorem{theorem}{Theorem}
\newtheorem{lemma}{Lemma}
\newtheorem{proposition}{Proposition}

\newtheorem{corollary}{Corollary}

\theoremstyle{remark}
\newtheorem{remark}{Remark}
\renewcommand{\phi}{\varphi}
\renewcommand{\epsilon}{\varepsilon}

\renewcommand{\Re}{\text{Re}}
\begin{document}
\title{Sums of compositions of pairs of projections}
\author{Andrzej Komisarski}
\address{Andrzej Komisarski\\
Department of Probability Theory and Statistics\\Faculty of Mathematics and Computer Science\\
University of \L\'od\'z\\ul.Banacha 22\\90-238 \L\'od\'z\\Poland}
\email{andkom@math.uni.lodz.pl}
\author{Adam Paszkiewicz}
\address{Adam Paszkiewicz\\
Department of Probability Theory and Statistics\\Faculty of Mathematics and Computer Science\\
University of \L\'od\'z\\ul.Banacha 22\\90-238 \L\'od\'z\\Poland}
\email{adampasz@math.uni.lodz.pl}
\thanks{This paper is partially supported by NCN Grant no. N N201 605840.}
\begin{abstract}
We give some necessary and sufficient conditions for the possibility to represent a Hermitian operator on an infinite-dimensional
Hilbert space (real or complex) in the form $\sum_{i=1}^nQ_iP_i$, where $P_1,\dots,P_n$, $Q_1,\dots,Q_n$ are orthogonal projections.
We show that the smallest number $n=n(c)$ admitting the representation $x=\sum_{i=1}^{n(c)}Q_iP_i$ for every $x=x^*$ with $\|x\|\leq c$
satisfies $8c+\frac83\leq n(c)\leq 8c+10$. This is a partial answer to the question asked by L. W. Marcoux in 2010.
\end{abstract}

\maketitle

\section{Introduction}

The research on representing an operator on the Hilbert space as a sum or a linear combination of orthogonal projections
(or idempotents, square-zero operators, commutators of projections and so on) has a long history.
We mention here important papers by Stampfli \cite{Stampfli} (who showed that every operator on infinite dimensional $H$ is a sum of 8 idempotents),
Fillmore \cite{Fillmore} (who showed that every operator on infinite dimensional $H$ is a sum of 64 square-zero operators and a linear combination
of 257 orthogonal projections) and Pearcy and Topping \cite{PT} (who improved these results showing that every operator on infinite dimensional $H$
is a sum of 5 idempotents, a sum of 5 square-zero operators and a linear combination of 16 orthogonal projections).
For a deep survey on this subject see an expository paper by Marcoux \cite{Marcoux}.

Note that the sum of orthogonal projections is always a positive operator.
For this reason if we want to represent any operator (or at least any self-adjoint operator) as a sum of operators
belonging to some class $\mathcal K\subset B(H)$ then we cannot restrict ourselves to the class of orthogonal projections
and we need to consider some other classes. In 2003 Bikchentaev \cite{Bikchentaev} showed that every operator $x$ on the infinite dimensional
Hilbert space $H$ is a sum of compositions of pairs of projections, i.e. $x=\sum_{i=1}^nQ_iP_i$ for some $n$ and orthogonal projections
$P_1,\dots,P_n$, $Q_1,\dots,Q_n$. Note that the assumption $\dim H=\infty$ is necessary because every operator on the finite-dimensional Hilbert
space has finite trace and the equality $x=\sum_{i=1}^nQ_iP_i$ implies $\text{trace}(x)=\sum_{i=1}^n\text{trace}(Q_iP_i)\geq0$.
To obtain his result Bikchentaev uses the representation of an operator as a sum of 5 idempotents (Pearcy--Topping \cite{PT})
but he does not estimate the number of summands in his representation. This problem is explicitly posed by Marcoux \cite{Marcoux}:
for any $c>0$ find possibly small $n(c)$ such that if $\|x\|\leq c$ then $x=\sum_{i=1}^{n(c)}Q_iP_i$ for some orthogonal projections
$P_1,\dots,P_{n(c)}$, $Q_1,\dots,Q_{n(c)}$.
The first attempt to answer this question for self-adjoint operators $x$ was presented in \cite{BikPasz}
where Bikchentaev and Paszkiewicz show that if $\|x\|\leq\frac1{20}$ then the considered representation needs at most 6 summands,
hence $n(c)\leq6\lceil20c\rceil\sim 120c$ (for the self-adjoint operators).
Now we extend the ideas presented in \cite{BikPasz} and we show that for the self-adjoint operators $x$
we have $8c+\frac83\leq n(c)\leq8c+10$ (hence $n(c)\sim 8c$ for large $c$), see Corollary \ref{cornc}.

Moreover, we have the following phenomenon. Let $c(n)$ and $C(n)$ be the largest positive numbers
such that the representation $x=\sum_{i=1}^nQ_iP_i$ is possible for any $x$ satisfying $0\leq x\leq C(n)\cdot\mathbf1$
or $-c(n)\cdot\mathbf1\leq x\leq 0$. Then $C(n)\approx 8c(n)$ for large $n$.
Thus it is natural to characterize the operators $x=x^*$ admitting the representation $x=\sum_{i=1}^nQ_iP_i$ using operator inequalities.
We give some simple and precise, necessary and sufficient conditions of that type
valid for both real and complex Hilbert spaces.
An important tool in our investigation is a description of the matrix representation of all possible compositions of pairs of projections
in $2$-dimensional Hilbert space (Lemma \ref{lemat}). We will also use the spectral theorem for the self-adjoint, bounded operators.

\section{Main results}

Now we present the main results of the paper.

\begin{theorem}\label{thm1}
Let $H$ be a real or complex Hilbert space and let $n$ be positive integer.
If $x=x^*\in B(H)$ satisfies $x=\sum_{i=1}^nQ_iP_i$
for some orthogonal projections $P_1$,\dots ,$P_n$, $Q_1$,\dots, $Q_n$ then 
$$-\frac n8\cdot\mathbf1\leq x\leq n\cdot\mathbf1.$$
\end{theorem}

It proves that the constants $-\frac n8$ and $n$ in this theorem cannot be improved.

\begin{proposition}\label{proposa}
The constant $n$ in Theorem \ref{thm1} cannot be decreased.
If $\dim H\geq2$ and $n$ is even then the constant $-\frac n8$ in Theorem \ref{thm1} cannot be increased.
\end{proposition}

If $n$ is odd then $-\frac n8$ can be replaced by some greater constant. However, we have not found its optimal value.

Theorem \ref{thm1} gives some conditions necessary for the representation $x=\sum_{i=1}^nQ_iP_i$.
The following Theorem shows that these conditions are not sufficient.

\begin{theorem}\label{twii}
Let $H$ be a real or complex Hilbert space and let $n$ be positive integer.
Suppose that $x=x^*\in B(H)$ satisfies $x\leq a\cdot\mathbf 1$ for some $a<-\frac{(n-2)^2}{8n}$.
Then $x\neq\sum_{i=1}^nQ_iP_i$ for every orthogonal projections $P_1$,\dots ,$P_n$, $Q_1$,\dots, $Q_n$.
\end{theorem}

Sufficient conditions are given in the next Theorem.

\begin{theorem}\label{twiii}
Let $H$ be a real or complex infinite dimensional Hilbert space and let $n\geq4$ be even.
If $x=x^*\in B(H)$ is an operator satisfying
$$-\frac{(n-4)^2}{8n}\cdot\mathbf1\leq x\leq(n-2)\cdot\mathbf1$$
then there exist orthogonal projections $P_1$,\dots ,$P_n$, $Q_1$,\dots, $Q_n$
such that $x=\sum_{i=1}^nQ_iP_i$.
\end{theorem}

As a consequence of Theorems \ref{twii} and \ref{twiii}
we obtain the following estimates for the constants $n(c)$ in the Morcoux's problem.

\begin{corollary}\label{cornc}
For every $c>0$ let $n(c)$ be the smallest number such that for every $x=x^*\in B(H)$, $\dim H=\infty$,
satisfying $\|x\|\leq c$ the representation $x=\sum_{i=1}^nQ_iP_i$ is possible.
Then we have
$$2+4c+4\sqrt{c^2+c}\leq n(c)\leq2\left\lceil 2+2c+2\sqrt{c^2+2c}\right\rceil.$$
In particular $8c+\frac83\leq n(c)\leq8c+10$, hence $\frac{n(c)}c\to8$ for $c\to\infty$.
\end{corollary}

\section{Proofs}

\begin{lemma}\label{lemat}
Let $\mathbb K=\mathbb R$ or $\mathbb C$, let $e_1=\left(\begin{smallmatrix}1\\0\end{smallmatrix}\right)$ and $e_2=\left(\begin{smallmatrix}0\\1\end{smallmatrix}\right)\in\mathbb K^2$ and let $A\subset\mathbb R^2$
be a set of all pairs $(\Re(QPe_1,e_1),\Re(QPe_2,e_2))$, where $P$ and $Q$ are one-dimensional projections
in $\mathbb K^2$. Then $A=\{(x,y)\in\mathbb R^2:(x-y)^2\leq x+y\leq 1\}$.
Moreover, there exist Borel functions $P^\cdot$ and $Q^\cdot:A\to B(\mathbb K^2)$
such that for every $(x,y)\in A$ the operators $P^{x,y}$ and $Q^{x,y}$ are one-dimensional projections,
$(Q^{x,y}P^{x,y}e_1,e_1)=x$ and $(Q^{x,y}P^{x,y}e_2,e_2)=y$.
\end{lemma}

\begin{proof}
Let $(x,y)\in A$, hence $x=\Re(QPe_1,e_1)$ and $y=\Re(QPe_2,e_2)$ for some one-dimensional projections
$P=\left(\begin{smallmatrix}p_1\\p_2\end{smallmatrix}\right)(\overline{p_1},\overline{p_2})$
and
$Q=\left(\begin{smallmatrix}q_1\\q_2\end{smallmatrix}\right)(\overline{q_1},\overline{q_2})$
with $p_1$, $p_2$, $q_1$, $q_2\in\mathbb K$ satisfying $\|(p_1,p_2)\|=\|(q_1,q_2)\|=1$.
Then
$$x=\Re\left((1,0)\left(\begin{smallmatrix}q_1\\q_2\end{smallmatrix}\right)(\overline{q_1},\overline{q_2})
\left(\begin{smallmatrix}p_1\\p_2\end{smallmatrix}\right)(\overline{p_1},\overline{p_2})\left(\begin{smallmatrix}1\\0\end{smallmatrix}\right)\right)
=\Re(q_1\overline{p_1}(\overline{q_1}p_1+\overline{q_2}p_2)),$$
$$y=\Re\left((0,1)\left(\begin{smallmatrix}q_1\\q_2\end{smallmatrix}\right)(\overline{q_1},\overline{q_2})
\left(\begin{smallmatrix}p_1\\p_2\end{smallmatrix}\right)(\overline{p_1},\overline{p_2})\left(\begin{smallmatrix}0\\1\end{smallmatrix}\right)\right)
=\Re(q_2\overline{p_2}(\overline{q_1}p_1+\overline{q_2}p_2)).$$
It follows that
$x+y=|q_1\overline{p_1}+q_2\overline{p_2}|^2\leq\|(q_1,q_2)\|^2\|(p_1,p_2)\|^2=1$ and
$$(x-y)^2\leq|(\overline{q_1}p_1+\overline{q_2}p_2)(q_1\overline{p_1}-q_2\overline{p_2})|^2\leq|\overline{q_1}p_1+\overline{q_2}p_2|^2\|(q_1,q_2)\|^2\|(p_1,-p_2)\|^2=x+y.$$

Now, let $(x,y)\in\mathbb R^2$ be such that $(x-y)^2\leq x+y\leq 1$.
If $(x,y)=(0,0)$ then we consider one-dimensional projections $P^{0,0}=\left(\begin{smallmatrix}1&0\\0&0\end{smallmatrix}\right)$
and $Q^{0,0}=\left(\begin{smallmatrix}0&0\\0&1\end{smallmatrix}\right)$ and we have
$(Q^{0,0}P^{0,0}e_1,e_1)=(Q^{0,0}P^{0,0}e_2,e_2)=0$.
Hence $(0,0)\in A$.
If $(x,y)\neq(0,0)$ then for $s:=x+y$ and $d:=x-y$ we have $s>0$, $s-d^2\geq0$ and $\tfrac1s-1\geq0$
and we can define
$$P^{x,y}=\begin{pmatrix}
\frac{1+d+\sqrt{(s-d^2)(\tfrac1s-1)}}2
&\frac{\sqrt{s-d^2}-d\sqrt{\tfrac1s-1}}2\\
\frac{\sqrt{s-d^2}-d\sqrt{\tfrac1s-1}}2
&\frac{1-d-\sqrt{(s-d^2)(\tfrac1s-1)}}2
\end{pmatrix},$$
$$Q^{x,y}=\begin{pmatrix}
\frac{1+d-\sqrt{(s-d^2)(\tfrac1s-1)}}2
&\frac{\sqrt{s-d^2}+d\sqrt{\tfrac1s-1}}2\\
\frac{\sqrt{s-d^2}+d\sqrt{\tfrac1s-1}}2
&\frac{1-d+\sqrt{(s-d^2)(\tfrac1s-1)}}2
\end{pmatrix}.$$
It is easy to check that $P^{x,y}=(P^{x,y})^*$, $Q^{x,y}=(Q^{x,y})^*$,
$\det(P^{x,y})=\det(Q^{x,y})=0$ and $\text{trace}(P^{x,y})=\text{trace}(Q^{x,y})=1$,
hence $P^{x,y}$ and $Q^{x,y}$ are one-dimensional projections.
Moreover $(Q^{x,y}P^{x,y}e_1,e_1)=x$ and $(Q^{x,y}P^{x,y}e_2,e_2)=y$,
hence $(x,y)\in A$.

The maps $A\ni(x,y)\mapsto P^{x,y}$ and $A\ni(x,y)\mapsto Q^{x,y}$
are continuous everywhere besides $(0,0)$, hence they are Borel maps, as required.
\end{proof}

\begin{corollary}\label{wniosek}
Let $\mathbb K=\mathbb R$ or $\mathbb C$. If $e\in\mathbb K^2$ satisfies $\|e\|=1$ and if $P$, $Q$ are one-dimensional projections
in $\mathbb K^2$ then $-\frac18\leq\Re(QPe,e)\leq 1$.
\end{corollary}

\begin{proof}
Without loss of generality (we can choose an appropriate coordinate system)
it is enough to consider the case $e=e_1=\left(\begin{smallmatrix}1\\0\end{smallmatrix}\right)$.
Then the set of possible values of $\Re(QPe,e)$ is $\{x:(x,y)\in A$ for some $y\in\mathbb R\}=(-\frac18,1)$.
\end{proof}

\begin{proof}[Proof of Proposition \ref{proposa}]
For any (real or complex) Hilbert space $H$ and $P_1=\dots=P_n=Q_1=\dots=Q_n=\mathbf 1$
we have $x=\sum_{i=1}^nQ_iP_i=n\cdot\mathbf 1$, hence the constant $n$ cannot be decreased.

Let $H=\mathbb R^2$ or $H=\mathbb C^2$ and let $n$ be even.
We put $$Q_1=Q_3=\dots=Q_{n-1}=Q^{-1/8,3/8},$$
$$P_1=P_3=\dots=P_{n-1}=P^{-1/8,3/8},$$
$$Q_2=Q_4=\dots=Q_n=\left(\begin{smallmatrix}1&0\\0&-1\end{smallmatrix}\right)Q^{-1/8,3/8}\left(\begin{smallmatrix}1&0\\0&-1\end{smallmatrix}\right),$$
$$P_2=P_4=\dots=P_n=\left(\begin{smallmatrix}1&0\\0&-1\end{smallmatrix}\right)P^{-1/8,3/8}\left(\begin{smallmatrix}1&0\\0&-1\end{smallmatrix}\right).$$
Since
$Q^{-1/8,3/8}P^{-1/8,3/8}=\left(\begin{smallmatrix}-\frac18&b\\c&\frac38\end{smallmatrix}\right)$ for some $b,c\in\mathbb R$,
we get that $x=\sum_{i=1}^nQ_iP_i=\left(\begin{smallmatrix}-\frac n8&0\\0&\frac {3n}8\end{smallmatrix}\right)$ is self-adjoint and
the constant $-\frac n8$ in Theorem \ref{thm1} cannot be increased.
For any $H$ with $\dim H\geq2$ the result easily follows from the two-dimensional case.
\end{proof}

\begin{proposition}\label{propos}
Let $K$ be a real or complex Hilbert space, $z_1,z_2\in B(K)$ be two self-adjoint commuting operators and let $z_1=\int x(\lambda)E(d\lambda)$ and $z_2=\int y(\lambda)E(d\lambda)$
be their spectral representations with a common spectral measure $E$.
Assume that for every $\lambda\in\mathbb R$ we have $(x(\lambda),y(\lambda))\in A$, where $A$ is the set defined in Lemma \ref{lemat}.
Then $z=z_1\oplus z_2\in B(K\oplus K)$ satisfies $2z=QP+Q'P'$ for some projections $P$, $Q$, $P'$ and $Q'$ in $K\oplus K$.
\end{proposition}

\begin{proof}
Using Lemma \ref{lemat}, for every $\lambda\in\mathbb R$ we obtain
$P^{x(\lambda),y(\lambda)}=\left(\begin{smallmatrix}p_{11}(\lambda)&p_{12}(\lambda)\\p_{21}(\lambda)&p_{22}(\lambda)\end{smallmatrix}\right)$
and
$Q^{x(\lambda),y(\lambda)}=\left(\begin{smallmatrix}q_{11}(\lambda)&q_{12}(\lambda)\\q_{21}(\lambda)&q_{22}(\lambda)\end{smallmatrix}\right)$,
where $p_{ij},q_{ij}:\mathbb R\to\mathbb R$ are Borel functions.
We define
$$P=\begin{pmatrix}
\int p_{11}(\lambda)E(d\lambda)&\int p_{12}(\lambda)E(d\lambda)\\
\int p_{21}(\lambda)E(d\lambda)&\int p_{22}(\lambda)E(d\lambda)
\end{pmatrix},
\qquad Q=\begin{pmatrix}
\int q_{11}(\lambda)E(d\lambda)&\int q_{12}(\lambda)E(d\lambda)\\
\int q_{21}(\lambda)E(d\lambda)&\int q_{22}(\lambda)E(d\lambda)
\end{pmatrix},$$
$$P'=\left(\begin{smallmatrix}\mathbf1&0\\0&-\mathbf1\end{smallmatrix}\right)P\left(\begin{smallmatrix}\mathbf1&0\\0&-\mathbf1\end{smallmatrix}\right)
\quad\text{and}\quad Q'=\left(\begin{smallmatrix}\mathbf1&0\\0&-\mathbf1\end{smallmatrix}\right)Q\left(\begin{smallmatrix}\mathbf1&0\\0&-\mathbf1\end{smallmatrix}\right).$$
Using Lemma \ref{lemat} and the von Neumann operator calculus we easily obtain that $P$, $Q$, $P'$ and $Q'$ are projections in $K\oplus K$ and
$$QP+Q'P'=\begin{pmatrix}
2\int x(\lambda)E(d\lambda)&0\\
0&2\int y(\lambda)E(d\lambda)
\end{pmatrix}=2z.$$
\end{proof}

\begin{proof}[Proof of Theorem \ref{thm1}]
Let $H$ be a Hilbert space over the field $\mathbb K$ with $\mathbb K=\mathbb C$ or $\mathbb R$.

If $\dim H=1$ then the only projections in $H$ are $\mathbf0$ and $\mathbf1$.
It follows that if $x=\sum_{i=1}^nQ_iP_i$ then $x=m\cdot\mathbf1$ for some $m=0,\dots,n$.
In the sequel we assume that $\dim H\geq 2$.

Now, we fix $e\in H$ and $i\in\{1,\dots,n\}$.
Let $p,q$ be one-dimensional projections such that $pe=P_ie$ and $qe=Q_ie$.
Moreover, let $r$ be two-dimensional projection satisfying $p\leq r$ and $q\leq r$
and let $U:\mathbb K^2\to H$ be an isometry satisfying $UU^*=r$.
Then
$$(Q_iP_ie,e)=(P_ie,Q_ie)=(pe,qe)=(rpre,qre)=(UU^*pUU^*e,qUU^*e)=(Pe',Qe')=(QPe',e'),$$
where $P=U^*pU$ and $Q=U^*qU$ are one-dimensional projections in $\mathbb K^2$
and $e'=U^*e\in\mathbb K^2$.
If $e'\neq0$, then
$$(Q_iP_ie,e)=(QPe',e')=\left(QP\frac{e'}{\|e'\|},\frac{e'}{\|e'\|}\right)\cdot\|e'\|^2.$$
Since $\|e'\|\leq\|e\|$ and (by Corollary \ref{wniosek}) $-\frac18\leq\Re\left(QP\frac{e'}{\|e'\|},\frac{e'}{\|e'\|}\right)\leq 1$
we obtain $-\frac18\cdot\|e\|^2\leq\Re(Q_iP_ie,e)\leq\|e\|^2$.
If $e'=0$, then $(Q_iP_ie,e)=(QPe',e')=0$ and the last inequality is also satisfied.

Summing the obtained inequalities with $i=1,\dots,n$
and using $\sum_{i=1}^n\Re(Q_iP_ie,e)=\Re(xe,e)=(xe,e)$ we get
$-\frac n8\cdot\|e\|^2\leq(xe,e)\leq n\cdot\|e\|^2$, which implies that the self-adjoint operator $x$ satisfies
$-\frac n8\cdot\mathbf1\leq x\leq n\cdot\mathbf1$.
\end{proof}

\begin{proof}[Proof of Theorem \ref{twii}]
Aiming at a contradiction we assume that $a<-\frac{(n-2)^2}{8n}\cdot\mathbf1$, $x=x^*\leq a\cdot\mathbf 1$
and $x=\sum_{i=1}^nQ_iP_i$ for some orthogonal projections $P_1$,\dots ,$P_n$, $Q_1$,\dots, $Q_n$.

For $i=1,\dots,n$ let $m_i=\inf\{\Re(Q_iP_ie,e):\|e\|=1\}$. Without loss of generality we may assume that $m_1=\min\{m_1,\dots,m_n\}$.
Clearly $nm_1\leq\sum_{i=1}^nm_i\leq\inf\{(xe,e):\|e\|=1\}\leq a$, hence $m_1\leq\frac an<-\frac{(n-2)^2}{8n^2}$.
We put $M=\sup\{\Re(Q_1P_1e,e):\|e\|=1\}$. We fix positive $\epsilon<\frac{(n-2)^2}{8n^2}$ and we choose $e\in H$ satisfying $\|e\|=1$
and $\Re(Q_1P_1e,e)>M-\epsilon$. We have
$$a\geq(xe,e)=\Re(Q_1P_1e,e)+\sum_{i=2}^n\Re(Q_iP_ie,e)>M-\epsilon+(n-1)m_1.$$
Next, we choose $f_1\in H$ satisfying $\|f_1\|=1$ and $\Re(Q_1P_1f_1,f_1)<m_1+\epsilon$.
Then for every $f_2\in H$ with $\|f_2\|=1$ one has
\begin{equation}\label{eqn}
a>M-\epsilon+(n-1)m_1\geq\Re(Q_1P_1f_2,f_2)+(n-1)\Re(Q_1P_1f_1,f_1)-n\epsilon.
\end{equation}

By $\Re(Q_1P_1f_1,f_1)<m_1+\epsilon<0$ we have that $P_1f_1\neq 0$ and $P_1f_1\neq f_1$,
hence $f_1$ and $P_1f_1$ are linearly independent. Let $r$ be the projection onto $\text{span }(f_1,P_1f_1)$
and let $p\leq r$ and $q$ be one-dimensional projections such that $pf_1=P_1f_1$ and $qp=Q_1p$.
The subspace $rH$ is isometric to $\mathbb R^2$ (or $\mathbb C^2$) and we are going to use Lemma \ref{lemat}.
We choose $f_2\in rH$ satisfying $f_2\perp f_1$ and $\|f_2\|=1$.
Since $pf_1=P_1f_1$ and $p(P_1f_1)=P_1(P_1f_1)$ it follows that $pf=P_1f$
for every $f\in rH$. In particular $pf_2=P_1f_2$, hence $qpf_2=Q_1P_1f_2$.

Note that $rqr$ is one-dimensional self-adjoint operator, hence $rqr=\alpha q'$
for some $0\leq\alpha\leq1$ and  one-dimensional projection $q'\leq r$. By \eqref{eqn} we have
\begin{equation*}
\begin{split}
a+n\epsilon&>\Re(Q_1P_1f_2,f_2)+(n-1)\Re(Q_1P_1f_1,f_1)=\Re(qpf_2,f_2)+(n-1)\Re(qpf_1,f_1)\\
&=\Re(qrpf_2,rf_2)+(n-1)\Re(qrpf_1,rf_1)=\Re(rqrpf_2,f_2)+(n-1)\Re(rqrpf_1,f_1)\\
&=\alpha\left[\Re(q'pf_2,f_2)+(n-1)\Re(q'pf_1,f_1)\right].
\end{split}
\end{equation*}
We have $a+n\epsilon<0$, hence $\Re(q'pf_2,f_2)+(n-1)\Re(q'pf_1,f_1)<0$.
Thus
\begin{equation}\label{eqnn}
a+n\epsilon>\alpha\left[\Re(q'pf_2,f_2)+(n-1)\Re(q'pf_1,f_1)\right]\geq \Re(q'pf_2,f_2)+(n-1)\Re(q'pf_1,f_1).
\end{equation}
On the other hand, by Lemma \ref{lemat} and an elementary computation concerning the set $A$ defined in that lemma we have
$$\Re(q'pf_2,f_2)+(n-1)\Re(q'pf_1,f_1)\geq\inf\{y+(n-1)x:(x,y)\in A\}=-\frac{(n-2)^2}{8n},$$
which contradicts \eqref{eqnn} for small enough $\epsilon$.
\end{proof}

\begin{remark}\label{remark}
Let $K$ be a Hilbert space, let $x=x^*\in B(K)$. Assume that cardinal numbers $d_1$, $d_2$ satisfy
$d_1+d_2=\dim K$. Then there exists a projection $E$ on $K$ such that $\dim E=d_1$,
$\dim (\mathbf 1-E)=d_2$ and $x$ commute with $E$.
\end{remark}

\begin{proof}[Proof of Theorem \ref{twiii}]
Let $n=2m\geq4$ be fixed. We will define self-adjoint operators $y_1,\dots,y_m$ satisfying $x=\sum_{i=1}^my_i$ and such that
$y_i=Q_iP_i+Q_i'P_i'$ for some projections $P_i$, $Q_i$, $P_i'$ and $Q_i'$ (then the proof will be finished).

We will use the following observation.
For $y=y*\in B(H)$ the existence of projections $P$, $Q$, $P'$ and $Q'$ satisfying $y=QP+Q'P'$
is a consequence of the following condition:
There exist projections $\widehat G_1$, $\widehat G_2$, $\widetilde G_1$ and $\widetilde G_2\in B(H)$ satisfying:
\begin{enumerate}
\item $\widehat G_1+\widehat G_2+\widetilde G_1+\widetilde G_2=\mathbf1$, $\dim\widehat G_1=\dim\widehat G_2$ and $\dim\widetilde G_1=\dim\widetilde G_2$,
\item $\widehat G_1$, $\widehat G_2$, $\widetilde G_1$ and $\widetilde G_2$ commute with $y$,
\item $y\widehat G_2=0$ and $0\leq y\widehat G_1\leq 2\cdot\widehat G_1$,
\item $y\widetilde G_2=2b\cdot\widetilde G_2$ and $2a\cdot\widetilde G_1\leq y\widetilde G_1\leq 2(1-b)\cdot\widetilde G_1$,
\end{enumerate}
where $a=-\frac{(m-2)(m+2)}{8m^2}$ and $b=\frac{(m-2)(3m-2)}{8m^2}$.

Indeed, by (i) we have $\dim\widehat G_1=\dim\widehat G_2$ and we may identify $\widehat K:\approx\widehat G_1H\approx \widehat G_2H$
and then we may treat the operators $z_1=\frac{y\widehat G_2}2=0$ and $z_2=\frac{y\widehat G_1}2$ as the self-adjoint operators in $B(\widehat K)$ (here we also use (ii)).
Clearly $z_1$ and $z_2$ commute, hence they have the spectral representations $z_1=\int x(\lambda)E(d\lambda)$ and $z_2=\int y(\lambda)E(d\lambda)$
with a common spectral measure $E$.
Clearly $x(\lambda)=0$ and (by (iii)) $0\leq y(\lambda)\leq 1$ for every $\lambda$.
It follows that for every $\lambda\in\mathbb R$ we have $(x(\lambda),y(\lambda))\in A$.
By Proposition \ref{propos} we obtain $y(\widehat G_1+\widehat G_2)=2(z_1\oplus z_2)=\widehat Q\widehat P+\widehat Q'\widehat P'$
for some projections $\widehat P$, $\widehat Q$, $\widehat P'$, $\widehat Q'\leq \widehat G_1+\widehat G_2$.

Similarly, using (i), (ii) and (iv), we obtain $y(\widetilde G_1+\widetilde G_2)=\widetilde Q\widetilde P+\widetilde Q'\widetilde P'$
for some projections $\widetilde P$, $\widetilde Q$, $\widetilde P'$, $\widetilde Q'\leq\widetilde G_1+\widetilde G_2$.
Indeed, after identification $\widetilde K:\approx\widetilde G_1H\approx\widetilde G_2H$ we have $\frac{y(\widetilde G_1+\widetilde G_2)}2=\int x(\lambda)E(d\lambda)\oplus\int y(\lambda)E(d\lambda)$
with $x(\lambda)=b$ and $a\leq y(\lambda)\leq 1-b$ (by (iv)).
It follows that $(x(\lambda),y(\lambda))\in A$ for every $\lambda\in\mathbb R$ (the special choice of the constants $a$ and $b$ plays a role here)
and by Proposition \ref{propos} we obtain $y(\widetilde G_1+\widetilde G_2)=\widetilde Q\widetilde P+\widetilde Q'\widetilde P'$.

Finally (by (i)) we have
$$y=y(\widehat G_1+\widehat G_2)+y(\widetilde G_1+\widetilde G_2)=QP+Q'P'$$
for the projections $P=\widetilde P+\widehat P$, $Q=\widetilde Q+\widehat Q$,
$P'=\widetilde P'+\widehat P'$ and $Q'=\widetilde Q'+\widehat Q'$.

It remains to define self-adjoint operators $y_1,\dots,y_m$ satisfying $x=\sum_{i=1}^my_i$
and (i)-(iv) for appropriate $\widehat G_1$, $\widehat G_2$, $\widetilde G_1$ and $\widetilde G_2$ (depending on $i$).
We start by picking projections $E_1,\dots,E_m$ in $H$ such that $\sum_{i=1}^mE_i=\mathbf 1$, $\dim E_i=\dim H$ and $E_i$ commutes with $x$ for every $i$.
(Here we use Remark \ref{remark} $m-1$ times.)
Next, we define $F=\text{supp }(x-2b\cdot\mathbf1)^+$ and $F^\perp=\mathbf1-F^+$ (here $y^+=(y+|y|)/2$ for $y=y^*$). Clearly $F$ and $F^\perp$
commute with $x$ and with projections $E_i$.

Next, for each $i$ we define $\widehat G_{i1}=(\mathbf1-E_i)F$ and $\widetilde G_{i1}=(\mathbf1-E_i)F^\perp$.
Then we apply Remark \ref{remark} for $K=E_iH$, $d_1=\dim\widehat G_{i1}$ and $d_2=\dim\widetilde G_{i1}$
(clearly $d_1+d_2=\dim(\mathbf1-E_i)=\dim E_i$). We obtain projections $\widehat G_{i2}$ and $\widetilde G_{i2}=E_i-\widehat G_{i2}$
commuting with $x$ and satisfying $\dim\widehat G_{i2}=d_1=\dim\widehat G_{i1}$ and $\dim\widetilde G_{i2}=d_2=\dim\widetilde G_{i1}$.
Clearly condition (i) is satisfied.

We have that $2m$ projections $\widehat G_{i1}=(\mathbf1-E_i)F$, $\widetilde G_{i1}=(\mathbf1-E_i)F^\perp$ (with $i=1,\dots,m$)
mutually commute, because $F, E_1,\dots,E_m$ commute.
$2m$ projections $\widehat G_{i2}$, $\widetilde G_{i2}$ are mutually orthogonal, hence they commute.
Finally, each of the projections $\widehat G_{i2}$, $\widetilde G_{i2}$ commute with $E_1,\dots,E_m$ and $x$ (hence $F$)
thus they commute with each of $2m$ projections $\widehat G_{i1}$, $\widetilde G_{i1}$.
It follows that each pair of $4m$ projections $\widehat G_{i1}$, $\widetilde G_{i1}$, $\widehat G_{i2}$ and $\widetilde G_{i2}$
(with $i=1,\dots,m$) commute.

We define $y_i$'s as follows:
\begin{equation}\label{defy}
y_i=2b\cdot\widetilde G_{i2}+\frac1{m-1}\cdot(\mathbf1-E_i)x-\frac{2b}{m-1}\cdot\sum_{j\neq i}\widetilde G_{j2}
\end{equation}
It is easy to verify that $x=\sum_{i=1}^my_i$ and $y_i$ commutes with $\widehat G_{i1}$, $\widetilde G_{i1}$, $\widehat G_{i2}$ and $\widetilde G_{i2}$
(hence(ii) is satisfied).

By \eqref{defy} we have
\begin{equation}\label{yiG}
y_i=0\cdot\widehat G_{i2}+\frac{x-2b\cdot D}{m-1}\cdot\widehat G_{i1}+2b\cdot\widetilde G_{i2}+\frac{x-2b\cdot D}{m-1}\cdot\widetilde G_{i1},
\end{equation}
where $D:=\sum_{j\neq i}\widetilde G_{j2}\leq\sum_{j\neq i}E_j=\widehat G_{i1}+\widetilde G_{i1}$ is a projection
and it commutes with $\widehat G_{i1}$ and $\widetilde G_{i1}$.

We will verify conditions (iii) and (iv).
By \eqref{yiG}, $x\leq (n-2)\cdot\mathbf 1=2(m-1)\cdot\mathbf 1$, $b>0$ and $D\widehat G_{i1}\geq0$ ($D\widehat G_{i1}$ is a projection) we obtain 
$$y_i\widehat G_{i1}=\frac{x-2b\cdot D}{m-1}\cdot\widehat G_{i1}=\frac{x\widehat G_{i1}}{m-1}-\frac{2b\cdot D\widehat G_{i1}}{m-1}\leq2\cdot\widehat G_{i1}.$$
Since $\widehat G_{i1}$ is a subprojection of $F$ (which is the support of $(x-2b\cdot\mathbf1)^+$) we obtain that $(x-2b\cdot\mathbf1)\widehat G_{i1}\geq0$ thus
$$y_i\widehat G_{i1}=\frac{x-2b\cdot D}{m-1}\cdot\widehat G_{i1}=\frac{x-2b\cdot\mathbf1}{m-1}\cdot\widehat G_{i1}+\frac{2b}{m-1}\cdot(\mathbf1-D)\widehat G_{i1}\geq0.$$
By \eqref{yiG} we also have $y_i\widehat G_{i2}=0$, hence (iii) is satisfied.

Since $\widetilde G_{i1}$ is a subprojection of $F^\perp$, hence $(x-2b\cdot\mathbf1)\widetilde G_{i1}\leq0$. Consequently (by \eqref{yiG})
$$y_i\widetilde G_{i1}=\frac{x-2b\cdot D}{m-1}\cdot\widetilde G_{i1}=\frac{2b}{m-1}\cdot\widetilde G_{i1}+\frac{x-2b\cdot\mathbf1}{m-1}\cdot\widetilde G_{i1}-\frac{2b}{m-1}\cdot D\widetilde G_{i1}\leq\frac{2b}{m-1}\cdot\widetilde G_{i1}
\leq 2(1-b)\cdot\widetilde G_{i1}.$$
Here we used the inequality $\frac{b}{m-1}\leq 1-b$, which is valid for $b=\frac{(m-2)(3m-2)}{8m^2}$.
By $x\geq-\frac{(n-4)^2}{8n}\cdot\mathbf1=-\frac{(m-2)^2}{4m}\cdot\mathbf1$ we obtain
\begin{equation*}
\begin{split}
y_i\widetilde G_{i1}&=\frac{x-2b\cdot D}{m-1}\cdot\widetilde G_{i1}=\frac x {m-1}\cdot\widetilde G_{i1}-\frac{2b}{m-1}\cdot\widetilde G_{i1}+\frac{2b}{m-1}\cdot(\mathbf1-D)\widetilde G_{i1}\\
&\geq-\frac{(m-2)^2}{4m(m-1)}\cdot\widetilde G_{i1}-\frac{2b}{m-1}\cdot\widetilde G_{i1}=a\cdot\widetilde G_{i1}.
\end{split}
\end{equation*}
By \eqref{yiG} we have $y_i\widetilde G_{i2}=2b\cdot\widetilde G_{i2}$, hence (iv) is satisfied.
\end{proof}

\begin{proof}[Proof of Corollary \ref{cornc}]
By Theorem \ref{twii} we have $c\leq\frac{(n(c)-2)^2}{8n(c)}$.
Solving this inequality on $n(c)$ we obtain $2+4c+4\sqrt{c^2+c}\leq n(c)$.

Now, let $n=2\left\lceil 2+2c+2\sqrt{c^2+2c}\right\rceil$.
Then $n\geq4$ is even and it satisfies $c\leq\frac{(n-4)^2}{8n}$.
Hence, by Theorem \ref{twiii}, we know
that every $x=x^*$ satisfying $\|x\|\leq c$ admits the representation $x=\sum_{i=1}^nQ_iP_i$.
Thus $n(c)\leq2\left\lceil 2+2c+2\sqrt{c^2+2c}\right\rceil$.

The second part of the corollary follows by the inequalities
$$\left\lceil2+4c+4\sqrt{c^2+c}\right\rceil\geq 8c+\frac83
\quad\text{and}\quad
2\left\lceil 2+2c+2\sqrt{c^2+2c}\right\rceil\leq8c+10\quad\text{for }c>0.$$
\end{proof}

\section{Final remarks}

We do not know any estimates for the number $n(c)$ for not necessarily Hermitian operators.
It seems that finding such estimates might be easier for complex Hilbert spaces.
This belief is based on the possibility to represent any operator as $x+iy$
with self-adjoint $x$ and $y$, which is possible only in the complex case.

Bikchentaev generalized his result about representation $x=\sum_{i=1}^nQ_iP_i$ in $B(H)$ to
wide classes of $C^*$-algabras, in particular he considered properly infinite von Neumann algebras
(\cite{Bikii}, \cite{Bikiii}). We believe that all the results proved in our paper can also be generalized from $B(H)$
to any properly infinite von Neumann algebra.


\begin{thebibliography}{12}
\bibitem{Bikchentaev} A. M. Bikchentaev, \emph{On the representation of linear operators in a Hilbert space as finite sums of products of projections},
Dokl. Akad. Nauk 393 (2003), no. 4, 444--447 (Russian); English transl.: Doklady Mathematical Sciences, 68 (2003), 376--379.


\bibitem{Bikii} A. M. Bikchentaev, \emph{On the representation of elements of a von Neumann algebra in the form of finite sums of products of projections},
Siberian Math. J. 46 (2005), no. 1, 24--34.

\bibitem{Bikiii} A. M. Bikchentaev, \emph{Representation of elements of von Neumann algebras in the form of finite sums of products of projections II}, Operator theory 20, 15--23,
Theta Ser. Adv. Math., 6, Theta, Bucharest, 2006.

\bibitem{BikPasz} A. M. Bikchentaev, A. Paszkiewicz, \emph{On representation of Hermitian operator as a sum of six products of two projections}, preprint.

\bibitem{Fillmore} P. A. Fillmore, \emph{Sums of operators with square zero}, Acta Sci. Math. (Szeged) 28 (1967) 285--288.

\bibitem{Marcoux} L. W. Marcoux, \emph{Projections, commutators and Lie ideals in $C^*$-algebras}, Math. Proc. R. Ir. Acad. 110A (2010), no. 1, 31--55.

\bibitem{PT} C. Pearcy, D. Topping, \emph{Sums of small numbers of idempotents}, Michigan Math. J. 14 (1967), 453--465.

\bibitem{Stampfli} J. G. Stampfli, \emph{Sums of projections}, Duke Math. J. 31 (1964) 455--461.
\end{thebibliography}
\end{document}